\documentclass{amsart}
\usepackage{amssymb}
\usepackage{amscd}
\usepackage[all,cmtip]{xy}
\usepackage{hyperref}
\usepackage{graphicx}

\newcommand{\comment}[1]{}

\newcommand{\Q}{\mathbf{Q}}
\newcommand{\Z}{\mathbf{Z}}

\newcommand{\F}{\mathbf{F}}

\newcommand{\m}{\mathfrak{m}}
\newcommand{\pa}{\mathfrak{p}}
\newcommand{\qa}{\mathfrak{q}}

\newcommand{\bi}{\langle\,\, ,\,\rangle}

\newcommand{\limplies}{\Longleftarrow}

\theoremstyle{plain}
\newtheorem{theorem}{Theorem}[section]

\newtheorem{corollary}[theorem]{Corollary}

\newtheorem{lemma}[theorem]{Lemma}

\newtheorem{proposition}[theorem]{Proposition}

\theoremstyle{definition}
\newtheorem{definition}[theorem]{Definition}
\newtheorem{remark}[theorem]{Remark}

\newtheorem{example}[theorem]{Example}

\makeindex

\begin{document}

\begin{center}
{\large\bf Michiel Kosters - Universiteit Leiden \par}
{\large\bf \texttt{mkosters@math.leidenuniv.nl}, \today \par}
 \vspace{3em} {\LARGE\bf Anisotropy and the integral closure \par} 
{\large\bf }
\end{center}
\section{Abstract}

Let $K$ be a number field and let $A$ be an order in $K$. The trace map from $K$ to $\Q$ induces a non-degenerate symmetric
bilinear form $\bi: B \times B \to \Q/\Z$ where $B$ is a certain finite abelian group of size $\Delta(A)$. In this article we discuss how one can
obtain information about $\mathcal{O}_K$ by purely looking at this symmetric bilinear form. The concepts of \emph{anisotropy} and
\emph{quasi-anisotropy}, as defined in another article by the author, turn out to be very useful. We will for example show that under certain
assumptions
one can obtain $\mathcal{O}_K$ directly from $\bi$. 

In this article we will work in a more general setting than we have discussed above. We consider orders over Dedekind domains. 

\section{Introduction}

We will discuss the relation between the new concepts of anisotropy and quasi-anisotropy as defined in \cite{KO2} and the integral closure of an order
in its total quotient ring. These concepts
show that in some cases one can find explicit formulas for the integral closure. \\

All rings in this article are assumed to be commutative. \\

We will first discuss some practical versions of the theorems which we will prove in this article. First let $B$ be a finite
abelian group, with additive notation. Then
we
define the \emph{lower root} of $B$ as
\begin{eqnarray*}
 \mathrm{lr}(B)=\sum_{r \in \Z} rB \cap B[r],
\end{eqnarray*}
where $B[r]=\{b \in B: rb=0 \}$. Let $\alpha$ be an algebraic integer and let $K=\Q(\alpha)$ and $A=\Z[\alpha]$. One has a trace map
$\mathrm{Tr}_{K/\Q}:
K \to \Q$. Now define the \emph{trace dual} of $A$ as 
\begin{eqnarray*}
 A^{\dagger}=\{x \in K: \mathrm{Tr}(xA) \subseteq \Z \}.
\end{eqnarray*}
It turns out that $A^{\dagger}$ contains $A$ and that $A^{\dagger}/A$ is a finite abelian group. Let $\overline{A}=\mathcal{O}_K$ be the integral
closure of $A$ in $K$. Our goal is to determine $\overline{A}$.  

The starting point of the theory which we develop in this article, is the following theorem (see Section \ref{ta}).

\begin{theorem} \label{1}
Let $p \in \Z$ be prime and assume that $p>[K:\Q]$. Then we have $p \mid \mathrm{exp} (\overline{A}/A)$ if and only if $p^2 \mid \mathrm{exp}
(A^{\dagger}/A)$, where $\mathrm{exp}$ stands for the exponent of the corresponding group. 
\end{theorem}

Using this theorem one can prove that, under the assumption that a certain form is anisotropic or quasi-anisotropic, $\mathcal{O}_K$ 
corresponds to the lower root of $A^{\dagger}/A$. For example, one has following theorem, which uses the concept of anisotropy (see Section \ref{ta}).

\begin{theorem} \label{2}
 Suppose that $B=A^{\dagger}/A$ and that $2 \nmid \#B$. Suppose that $B \cong \Z/m\Z \times \Z/m'\Z$, where $m=\prod_{p \mathrm{\ prime}} p^{n(p)}$
and similarly $m'=\prod_{p \mathrm{\ prime}} p^{n'(p)}$ such that for all primes $p$ we have $n(p)n'(p)=0$ or $n(p)+n'(p)$ is odd.
Then $\overline{A}/A=\mathrm{lr}(A^{\dagger}/A)$. 
\end{theorem}

Using quasi-anisotropy (and some other techniques) one finds a stronger version. In order to find the integral closure, 
it is enough to find the integral closure locally. We have the following theorem (see Section \ref{beba}).

\begin{theorem} \label{3}
Let $\m \subset A$ be a maximal ideal, let $p\Z=\m \cap \Z$. Define the numbers $n(i)$ such that $\left( A^{\dagger}/A \right)_{\m} \cong
\bigoplus_{i \geq 1} \left(\Z/p^i \Z \right)^{n(i)}$. Suppose that the following conditions are satisfied. 
\begin{enumerate}
\item
$p> \sum_{i \geq 1} n(i)$;
\item
There exist $i_1,i_2 \in \Z_{\geq 1}$ such that 
\begin{itemize}
\item
$i_1 \not \equiv i_2 \bmod{2}$;
\item
$n(i)=0$ for all $i \not \in \{1,i_1,i_2\}$;
\item
$n(i) \in \{0,[A/\m:\Z/p\Z]\}$ for $i \in \{i_1,i_2\}$.
\end{itemize}
\end{enumerate}
Then one has $\left( \overline{A}/A \right)_{\m}=\mathrm{lr}\left((A^{\dagger}/A)_{\m} \right)$.
\end{theorem}

In this article we will prove the results above in a more general case: we will work with orders over Dedekind domains.

\section{Tameness}
The concept of tameness will play an important role in later sections. In this section we fix a field $k$ and let $A$ be commutative finite
$k$-algebra. Recall
that $A$ is an artinian ring and has only finitely many maximal ideals (\cite{AT}, Chapter 8). We have a natural trace map
\begin{eqnarray*}
\mathrm{Tr}_{A/k}: A &\to& k \\
x &\mapsto& \mathrm{Tr}(\cdot x)
\end{eqnarray*}
where $\cdot x: A \to A$ is the multiplication by $x$ map and $\mathrm{Tr}(\cdot x)$ is the standard trace of an endomorphism on a vector space over $k$.

\begin{definition} \label{eta}
Consider the symmetric $k$-bilinear form 
\begin{eqnarray*}
\bi: A \times A &\to& k \\
\langle x,y \rangle &\mapsto& \mathrm{Tr}_{A/k}(xy).
\end{eqnarray*}
The radical of this form, $A^{\perp}$, is defined to be $\{x \in A: \mathrm{Tr}(xA)=0\}$.
We say that $A$ is \emph{tame} over $k$ if $A^{\perp}$ is equal to the nilradical of $A$. If $A$ is not tame, it is called \emph{wild}. 

We say that $A$ is \emph{finite \'etale} if $\bi$ is non-degenerate, that is, if the natural map $A \to \mathrm{Hom}(A,k)$ which maps $x \in A$ to
$\langle x,\ \rangle$ is an isomorphism of $k$-modules. This is easily seen to be equivalent to saying that the discriminant $\Delta(A/k)$ of $A$ over
$k$ is nonzero. Another equivalent notion is that $A$ is isomorphic to a finite product of finite separable field extensions of $k$
(\cite{LE}, Theorem 2.7).
\end{definition}

\begin{remark}
In Definition \ref{eta} it is always true that the nilradical is contained in the radical of $\bi$.
\end{remark}

\begin{definition}
A prime $\pa \subset A$ is called \emph{wild} if $(A/\pa)/k$ is inseparable or $\mathrm{char}(k) \mid \mathrm{length}_{A_{\pa}}(A_{\pa})$. If $\pa$ is
not wild, it is called \emph{tame}.
\end{definition}

\begin{proposition} \label{weed}
The radical of the algebra $A$ is equal to the intersection of all tame primes. The algebra $A$ is tame if and only if all primes of $A$ are tame.
Furthermore, if $A$ is wild, then $\mathrm{dim}_k(A) \geq \mathrm{dim}_k(A^{\perp})\geq \mathrm{char}(k)>0$.  
\end{proposition}

For the proof, we need the following lemma. 

\begin{lemma}[Trace formula] \label{tracef}
We have
\begin{eqnarray*}
\mathrm{Tr}_{A/k}(x)= \sum_{\pa \in \mathrm{Spec}(A)} e_{\pa} \cdot \mathrm{Tr}_{(A/\pa)/k}(x+\pa)
\end{eqnarray*}
where $e_{\pa}=\mathrm{length}_{A_{\pa}}(A_{\pa})$. 
\end{lemma}
\begin{proof}
First suppose that $(A,\pa)$ is a local ring. Since $A$ satisfies the descending and ascending chain conditions, there is a composition series $A=M_0
\supseteq M_1 \supseteq \ldots \supseteq M_{e_{\pa}}=0$ where the $M_i$ are $A$-modules and $M_i/M_{i-1} \cong A/\pa$ (see \cite{AT}, Proposition
6.8). As the trace is additive on exact sequences, we find
\begin{eqnarray*}
\mathrm{Tr}_{A/k}(x) &=& \sum_{i=1}^{e_{\pa}} \mathrm{Tr}_{(M_i/M_{i-1})/k}(x).
\end{eqnarray*}
Since we have isomorphisms $M_i/M_{i-1} \cong A/{\pa}$, all the multiplication maps by $x$ have the same trace. This shows that
$\mathrm{Tr}_{A/k}(x)=e_{\pa} \cdot \mathrm{Tr}_{(A/\pa)/k}(x+\pa)$. 

Now we will do the general case. We know that $A \cong \prod_{\pa \in \mathrm{Spec}(A)} A_{\pa}$ (see \cite{LA}, Exercise 10.9f). As $A/\pa \cong
A_{\pa}/\pa A_{\pa}$ by the natural map, we obtain 
\begin{eqnarray*}
\mathrm{Tr}_{A/k}(x) = \sum_{\pa \in \mathrm{Spec}(A)} e_{\pa} \cdot \mathrm{Tr}_{(A_{\pa}/\pa A_{\pa})/k}(x+\pa A_{\pa}) = \sum_{\pa \in
\mathrm{Spec}(A)} e_{\pa} \cdot \mathrm{Tr}_{(A/\pa)/k}(x+\pa).
\end{eqnarray*} 
\end{proof}

\begin{proof}[Proof of Proposition \ref{weed}]
The first statement is obtained from Lemma
\ref{tracef} and Proposition
3.4 from \cite{LO}.
For the second statement, use the Chinese remainder theorem to see that the nilradical, which is equal to the intersection of all prime ideals, is not
equal to the intersection
of a strict subset of the set of prime ideals. 
For the third statement, notice first of all that if $\mathrm{char}(k)=0$, then $A$ is tame. Suppose that $\pa$ is a wild prime, then we have
$\mathrm{dim}_k(A^{\perp}) \geq \mathrm{dim}_{k}(A_{\pa})$. We have $\mathrm{dim}_k(A_{\pa})=e_{\pa} \cdot \mathrm{dim}_k(A/\pa)$. As $\pa$ is wild,
either $e_{\pa}$ is divisible by $\mathrm{char}(k)$ or $A/\pa$ is an inseparable extension, with degree is divisible by $\mathrm{char}(k)$.
\end{proof}

\begin{lemma} \label{hu}
Suppose that $k \subseteq k' \subseteq A$ where $k'$ is a field. If $A$ is tame over $k$, then $A$ is tame over $k'$.  
\end{lemma}
\begin{proof}
 An element in the trace radical with respect to $k'$ will be in the trace radical with respect to $k$, hence will be nilpotent by definition of
tameness. 
\end{proof}

\section{Orders}

In this section let $R$ be a Dedekind domain. 

\begin{theorem} \label{superb}
Let $M$ be a finitely generated $R$-module. Then the following statements are equivalent:
\begin{enumerate}
\item
$M$ is torsion-free;
\item
$M$ is flat;
\item 
$M$ is projective.
\end{enumerate}
\end{theorem}
\begin{proof}
i $\Longleftrightarrow$ ii: See \cite{AT}, Exercise 9.5. 

i $\Longleftrightarrow$ iii: See \cite{MAY}, Theorem 7.2. 
\end{proof}

\begin{definition}
Let $T$ be a ring. Let $S=\{ x \in T\ |\ \mathrm{Ann}_T(x)=0 \} \subseteq T$ be the set consisting of elements that are not a zero divisor. Then we
define the \emph{total quotient ring} of $T$ as $Q(T)=S^{-1}T$.
\end{definition}

\begin{theorem} \label{quotient}
Let $T$ be a domain and let $A$ be an $T$-algebra that is torsion-free as $T$-module and integral over $T$. Then $A \otimes_T Q(T)=Q(A)$.
\end{theorem}
\begin{proof}
Assume that $A \neq 0$. Then $T=T \cdot 1 \subseteq A$ as $A$ is torsion-free. Now let $S=T \setminus \{0\}$ be the set of nonzero divisors of $T$ and
let $S'$ be the set of nonzero divisors of $A$. Then $S \subseteq S'$ as $A$ is torsion-free. We claim that $S'$ is the saturation of $S$
(see \cite{AT}, Exercise 7). We have to show that for any $x \in S'$, there exists $y \in A$ with $xy \in S$. Let $x \in S'$. As $A$ is integral over
$R$, it follows that $x^n+r_{n-1} x^{n-1}+\ldots+r_0=0$ for some $r_i \in T$. Assume that this relation is of minimal degree. We have $r_0 \neq 0$ as
$x$ is not a zero divisor. But this means that $x(x^{n-1}+r_{n-1}x^{n-2}+\ldots+r_1)=-r_0 \in T \setminus \{0\}=S$. Hence $S^{-1}A=S'^{-1}A$ and we
find (using \cite{AT}, Proposition 3.5)
\begin{eqnarray*}
A \otimes_T Q(T) = A \otimes_T S^{-1}T = S^{-1}A = S'^{-1}A= Q(A).
\end{eqnarray*} 
\end{proof}

\begin{definition} \label{ordero}
Let $R$ be a Dedekind domain. Let $A$ be an $R$-algebra. Then $A$ is called an \emph{order} over $R$ \index{order} if $A$ is finitely generated
torsion-free as an $R$-module and $Q(A)=A \otimes_R Q(R)$ is a finite  \'etale algebra over $Q(R)$.
\end{definition}

\begin{definition} \label{dua}
Let $R$ be a Dedekind domain and let $A$ over $R$. Let $M \subset Q(A)$ be a finitely generated $R$-module. Then we define the \emph{trace dual} of
$M$ to be the $R$-module \index{trace dual} 
\begin{eqnarray*}
M^{\dagger}=\{x \in Q(A): \mathrm{Tr}_{Q(A) /Q(R)}(xM) \subseteq R \}. 
\end{eqnarray*}
\end{definition}

\begin{definition}
Let $A$ be a ring. Then we define the \emph{integral closure} \index{integral closure} of $A$ in $Q(A)$ as 
\begin{eqnarray*}
\overline{A}=\{a \in Q(A): \mathrm{there\ is\ a\ monic\ } f\in A[x]: f(a)=0 \}.
\end{eqnarray*}
\end{definition}

\begin{lemma} \label{longi}
Let $R$ be a Dedekind domain and let $A \neq 0$ be an order over $R$. Then the following all hold:
\begin{enumerate}
\item
$R \subseteq A$ is integral and $\overline{A}$ is the integral closure of $R$ inside $Q(A)$;
\item
every order $B \subset Q(A)$ satisfies $\mathrm{Tr}_{Q(A)/Q(R)}(B) \subseteq R$;
\item
$A \subseteq \overline{A} \subseteq A^{\dagger}$;
\item
$A^{\dagger}$ is a finitely generated $R$-module and $A^{\dagger}/A$ is torsion as an $R$-module;
\item
$\overline{A}$ is the unique maximal element (under inclusion) of the set of orders $B \subseteq Q(A)$. 
\end{enumerate}
\end{lemma}
\begin{proof}
i. We have $R \subseteq A$ as $A \neq 0$ is torsion-free. As $A$ is finitely generated over $R$, we can apply Proposition 5.1 from \cite{AT} to see that $R \subseteq A$ is integral. The second statement follows from \cite{AT}, Corollary 5.4. 
 
ii. After enlarging $B$ if necessary, we may assume that $B \otimes Q(R)=Q(A)$. In this case the restriction of $\mathrm{Tr}_{Q(A)/Q(R)}$ to $B$ is
equal to the natural trace map to $R$ on $\mathrm{Hom}_R(B,R) \otimes_R B \cong \mathrm{End}_R(B)$ (see
\cite{LE}, 4.8).

iii, iv. First notice that we have a map $A^{\dagger} \to \mathrm{Hom}_R(A,R)$, that maps $x$ to $(y \mapsto \mathrm{Tr}_{Q(A)/Q(R)}(xy))$, and this
map is injective as $Q(A)$ is an \'etale $Q(R$)-algebra. Notice that $\mathrm{Hom}_R(A,R)$ is a noetherian $R$-module, as $R$ is noetherian and
$\mathrm{Hom}_R(A,R)$ is a finitely generated $R$-module. Hence $A^{\dagger}$ is finitely generated over $R$. Let $x \in \overline{A}$. Then $A[x]$ is
an order, hence \[ \mathrm{Tr}_{Q(A)/Q(R)}(A[x]) \subseteq R \] and $x \in A^{\dagger}$. As $Q(A)/A$ is torsion we obtain iii and iv.  

v. As all orders are integral over $R$, they are contained in the integral closure of $R$ in $Q(A)$, which is just $\overline{A}$. As $\overline{A}
\subseteq A^{\dagger}$ by ii, it follows that $\overline{A}$ is finitely generated and torsion-free. Also $Q(\overline{A})=Q(A)$ and hence
$\overline{A}$ is an order as well.

\end{proof}

\begin{definition} \label{tame2}
Let $R$ be a Dedekind domain and let $A \supseteq R$ be an order over $R$. Let $\pa \subset R$ be a nonzero prime. Then $A$ is said to be \emph{tame} at $\pa$ \index{prime!tame}\index{tame prime} if $A/\pa A$ is tame as an $R/\pa$-algebra. If $A$ is not tame at $\pa$, it is called \emph{wild} at $\pa$.\index{prime!wild}\index{wild prime}
\end{definition}

\begin{example}
Let $K \supset \Q$ be a number field. Let $p \in \Z$ be prime. Then according to the usual definition of tameness, $p$ is called tame if $p \nmid
e(\pa/p)$ for all $\pa \in \mathrm{Spec}(\mathcal{O}_K)$ with $\pa \cap \Z=(p)$. Here $e(\pa/p)$ is defined by $p \mathcal{O}_K=\prod_{\pa \in
\mathrm{Spec}(\mathcal{O}_K): \pa \cap \Z=(p)} \pa^{e(\pa/p)}$. We then find by the Chinese remainder theorem
\begin{eqnarray*}
\mathcal{O}_K/p \mathcal{O}_K &\cong& \prod_{\pa \in \mathrm{Spec}(\mathcal{O}_K): \pa \cap \Z=(p)} \mathcal{O}_K/ \pa^{e(\pa/p)}.
\end{eqnarray*}
From this last expression we deduce that $e_{\pa/p \mathcal{O}_K}=e(\pa/p)$. As $\F_p$ is a perfect field, we see that the two definitions of tameness are the same in this case. 
\end{example}

\section{Orders and localization}
In many of the coming theorems, it is useful to focus on only one prime $\pa \subset R$. This is why we use the notion of localization. 
We have the following lemma, which summarizes the situation. The proof of this lemma follows easily from the properties of localization (see
\cite{AT}).

\begin{lemma} \label{loop}
The following assertions all hold.
\begin{enumerate}
\item $A_{\pa}$ is an order over $R_{\pa}$;
\item $Q(A_{\pa})=Q(A)$;
\item $\left( A^{\dagger} \right)_{\pa}=\left( A_{\pa} \right)^{\dagger}$ and $\left(A^{\dagger}/A \right)_{\pa} \cong \left( A_{\pa}
\right)^{\dagger} / A_{\pa}$ as $R_{\pa}$-modules;
\item $\overline{A_{\pa}}= \overline{A}_{\pa}$ and $\left( \overline{A}/A \right)_{\pa} \cong \overline{A_{\pa}}/A_{\pa}$ as
$R_{\pa}$-modules;
\item $A$ is tame at $\pa$ if and only if $A_{\pa}$ is tame at $\pa R_{\pa}$. 
\end{enumerate}
\end{lemma}

\begin{lemma} \label{floop}
 We have $\overline{A}/A \cong \bigoplus_{\pa \in \mathrm{MaxSpec}(R)} \left( \overline{A_{\pa}}/A_{\pa} \right)$. 
\end{lemma}
\begin{proof}
 The assumptions of Theorem 2.13 of \cite{EI} are satisfied as $\overline{A}/A$ has finite length by Lemma \ref{longi} iv. Hence $\overline{A}/A
\cong 
\bigoplus_{\pa \in \mathrm{MaxSpec}(R)} \left( \overline{A}/A\right)_{\pa}$. Now use Lemma \ref{loop}.
\end{proof}

The strength of the previous lemma is that it suffices to find the integral closure locally, and
glue those local parts together to get the global integral closure. 

Assume in the rest of the
article that $R$ is local with maximal ideal $\pa$, that is, $R$ is a discrete valuation ring with
maximal ideal $\pa=(\pi)$, unless stated otherwise explicitly.  

\section{Orders and completion}

Some proofs become a lot clearer if our order $A$ is also local. This is one of the reasons why we use completions. Later we will see another
reason for using completions. Recall that $R$ is assumed to be a
discrete valuation ring with maximal ideal $\pa$ and that $A$ is an order over $R$. Let $\hat{R}$ be the completion
of $R$ with respect to its unique maximal ideal (see \cite{AT}, Chapter 10). Then $\hat{R}$ is a complete discrete valuation ring with maximal ideal
$\pa \hat{R}$. 
We have the following lemma, which shows that completion behaves nicely with respect to the
integral closure, trace duals and other things. The proof is routine and left to the reader. The reader
who wants to see the proofs can look at \cite{KO1}.

\begin{lemma} \label{cheat}
The following statements hold.
\begin{enumerate}
\item
The natural map $Q(R)/R \to Q(\hat{R})/\hat{R}$ is an
isomorphism.
\item

$A \otimes_R \hat{R}$ is an order over $R$;
\item
$\overline{A \otimes_R \hat{R}}=\overline{A} \otimes_{R} \hat{R}$; 
\item
$\overline{A}/A \cong \overline{A \otimes_R \hat{R}} / A \otimes_{R} \hat{R}$ as $\hat{R}$-modules by the natural map;
\item
$A^{\dagger} \otimes_{R} \hat{R} \cong \left( A \otimes_R \hat{R} \right)^{\dagger}$.
\item
$A^{\dagger}/A \cong \left( A \otimes_R \hat{R} \right)^{\dagger}/ A \otimes_R \hat{R}$ as $\hat{R}$-modules by the natural map;
\item
We have the following commutative diagram where the vertical maps are the natural maps and the horizontal maps look like $(\overline{x},\overline{y})
\mapsto \overline{\mathrm{Tr}(xy)}$ for the trace map on $Q(A \otimes_R \hat{R})$ respectively $Q(A)$:
\[
\xymatrix{
\left( A \otimes_R \hat{R} \right)^{\dagger}/A \otimes_R \hat{R} \times \left( A \otimes_R \hat{R} \right)^{\dagger}/A \otimes_R \hat{R} \ar[r] &
Q(\hat{R})/\hat{R} \\
A^{\dagger}/A  \times A^{\dagger}/A \ar[r] \ar[u] & Q(R)/R. \ar[u]
}  
\]
\item
The order $A$ is tame at $\pa$ if and only if $A \otimes_{R} \hat{R}$ is tame at $\pa \hat{R}$. 
\end{enumerate}
\end{lemma}

The reason to use this completion is the following theorem. 

\begin{theorem} \label{ks}
Let $A$ be an order over a complete discrete valuation ring $R$. Then the order $A$ has only finitely many maximal ideals and the localization
$A_{\m}$ at a maximal ideal $\m \subset R$ is a local order over $R$, which is complete with respect to its maximal ideal. Furthermore we have an
isomorphism $A \cong \prod_{\m \in \mathrm{Maxspec}(A)} A_{\m}$ as rings by the natural map. 
\end{theorem}
\begin{proof}
Corollary 7.6 from \cite{EI} tells us that there are only finitely maximal ideals and the localization $A_{\m}$ at a maximal ideal $\m \subset R$ is a
complete local ring which is finite over $R$, and $A \cong \prod_{\m \in \mathrm{Maxspec}(A)} A_{\m}$. As $A$ is projective over $R$ and direct
summands of projective modules are projective, it follows that the $A_{\m}$ are also projective over $R$. Now notice that $A \otimes_R Q(R)= \prod_{\m
\in \mathrm{Maxspec}(A)} \left( A_{\m} \otimes_R Q(R) \right)$ and hence \[ \Delta(Q(A)/Q(R))=\prod_{\m \in \mathrm{Maxspec}(A)}
\Delta(Q(A_{\m})/Q(R)).\] As $\Delta(Q(A)/Q(R)) \neq 0$, it follows that $\Delta(Q(A_{\m})/Q(R)) \neq 0$, which shows that these $A_{\m}$ are orders
over $R$.
\end{proof}

\begin{lemma} \label{mars}
Let $A \subseteq A'$ be an order over a complete discrete valuation ring $R$ with the same total quotient ring. Let $\m \subset A$ be maximal. Then
$A'_{\m}$ is an order over $R$ and $A_{\m} \subseteq A'_{\m} \subseteq Q(A_{\m})$. Furthermore, we have $A'_{\m} = \prod_{\m' \supseteq \m A'}
A'_{\m'}$ and $A'=\prod_{\m \in \mathrm{Maxspec}(A)} A'_{\m}$.
\end{lemma}
\begin{proof}
 Notice that $A'_{\m}=A' \otimes_A A_{\m}$ is still a finite $R$-algebra as $A' \otimes_R A_{\m}$ is. One applies Corollary 7.6 from \cite{EI} to see
that $A'_{\m} = \prod_{\m' \supseteq \m A'} A'_{\m'}$, and as the $A'_{\m'}$ are orders by Theorem \ref{ks}, it follows that $A'_{\m}$ is an order.
Now
notice that $A_{\m} \subseteq A'_{\m}=A' \otimes_A A_{\m} \subseteq Q(A) \otimes_A A_{\m}=Q(R) \otimes_R A \otimes_A A_{\m}=Q(R) \otimes
A_{\m}=Q(A_{\m})$ as required. The last statement follows from theorem \ref{ks}:
\begin{eqnarray*}
 A'=\prod_{\m' \in \mathrm{Maxspec}(A')} A'_{\m'}=\prod_{\m \in \mathrm{Maxspec}(A)} A'_{\m}.
\end{eqnarray*}
\end{proof}

\section{Going local directly}

In this section let $(R,\pa)$ be a discrete valuation ring and let $A$ be an order over $R$. 

\begin{lemma} \label{cheata}
We have $A \otimes_R \hat{R} \cong \prod_{\m \supseteq \pa A} \hat{A_{\m}}$ and the $\hat{A_{\m}}$ are local orders over $\hat{R}$ which are
complete with respect to its maximal ideal. 
\end{lemma}
\begin{proof}
As $A/\pa A$ is an artinian ring, we can write $\pa A \supseteq \prod_{\m \supseteq \pa A} \m^s$ for some fixed $s$ (see \cite{AT}, Chapter
8). But then
we have
by the Chinese remainer theorem
\begin{eqnarray*}
 A \otimes_R \hat{R} &=& \underset{\underset{i}{\leftarrow}}{\lim} A/\pa^i A = \underset{\underset{i}{\leftarrow}}{\lim} \prod_{\m \supseteq \pa A}
A/\m^{si}= \prod_{\m \supseteq
\pa A} \underset{\underset{i}{\leftarrow}}{\lim} A/\m^{i} \\
&=& \prod_{\m \supseteq \pa A} \underset{\underset{i}{\leftarrow}}{\lim} A_{\m}/\m^{i}A_{\m}= \prod_{\m \supseteq \pa A} \hat{A_{\m}}.
\end{eqnarray*}
Hence by Theorem \ref{ks} we see that $\hat{A_{\m}}$ is a local order over $\hat{R}$ which is complete with respect to its
maximal ideals.
\end{proof}

\begin{lemma} \label{sor}
 Let $\m$ be a maximal ideal of $A$. Then $\left( A^{\dagger}/A \right)_{\m} \cong \hat{A_{\m}}^{\dagger}/\hat{A_{\m}}$ as $A$-modules.
\end{lemma}
\begin{proof}
 As $A^{\dagger}/A$ is a module of finite length over $R$, it is a module of finite length over $A$. By Theorem 2.13 from \cite{EI} we have
$A^{\dagger}/A= \bigoplus_{\m \in \mathrm{MaxSpec}(R)} \left(A^{\dagger}/A \right)_{\m}$. By By Lemma \ref{cheat} and Lemma \ref{cheata} we have that
$A^{\dagger}/A \cong \bigoplus_{\m \in \mathrm{MaxSpec}(R)}
\hat{A_{\m}}^{\dagger}/\hat{A_{\m}}$. We have to show that both decomposition coincide. But $\hat{A_{\m}}^{\dagger}/\hat{A_{\m}}$ is a module over
$A_{\m}$ and hence the decompositions must coincide. 
\end{proof}

\section{Local orders}
Recall that $R$ is assumed to be local with maximal ideal $\pa$. As we have seen in the two previous sections, by completing one
obtains local orders. So let $A$ be a local order over $R$ with maximal ideal $\m$. 

\begin{lemma} \label{sop}
The ring $\m:\m=\{x \in Q(A): x\m \subseteq \m \}$ is an order over $R$ and $A \subseteq \m:\m \subseteq \overline{A}$.  
\end{lemma}
\begin{proof}
Let $x \in \m:\m$. Since $R$ is noetherian, $\m$ is a finitely generated $R$-module. As $A$ is torsion-free, $\m$ is a faithful $R$-module. Now apply Proposition 5.1 iii from \cite{AT} to see that $x$ is integral over $A$. Hence $\m:\m \subseteq \overline{A}$. We see that $\m:\m$ is finitely generated as an $R$-module and still torsion-free, as it is contained in $Q(A)$. As $Q(\m:\m)=Q(A)$, it follows that $\m:\m$ is an order over $R$. 
\end{proof}

The following theorem gives some equivalent criteria for testing if $A=\overline{A}$. 

\begin{theorem} \label{wuffer}
The the following statements are equivalent.
\begin{enumerate}
\item
$A=\overline{A}$; 
\item
$A=\m:\m$;
\item
$\m (A:\m)=A$;
\item
$\m:\m \neq A:\m$;
\item
$\m$ is principal;
\item
$A$ is a discrete valuation ring. 
\end{enumerate}
\end{theorem}
\begin{proof}

We first make a few remarks. Recall that $\overline{A}/A$ is a finitely generated torsion $R$-module (Lemma \ref{longi}). Let $r \in \Z_{\geq 0}$ such that $\pa^r \overline{A} \subseteq A$. As $A/\pa A$ is an artinian ring, it follows that $\m^n \subseteq \pa A$ for some $n \in \Z_{\geq 0}$. Hence there exists $s \in \Z_{\geq 1}$ such that $\m^s \overline{A} \subseteq A$.

Now we will prove that $A:\m \supsetneq A$. Suppose that $A:\m=A$. Pick $n \in \Z_{\geq 1}$ minimal such that $\m^n \subseteq \pa A$. But then $\m^{n-1} \subseteq \pa A:\m=\pa(A:\m)=\pa A$ (as $\pa$ is a principal ideal), a contradiction. 

i $\implies$ ii: This follows from Lemma \ref{sop}.

ii $\implies$ iii: We have $\m \subseteq \m(A:\m) \subseteq A$. Suppose that $\m(A:\m) \neq A$, then $\m(A:\m)=\m$. Using this and the second remark, we conclude that $\m:\m =A:\m \supsetneq A$, a contradiction. 

iii $\iff$ iv: Notice that $\m:\m \subseteq A:\m$. We have $\m(A:\m)=A$ iff $\m(A:\m) \neq \m$ iff $A:\m \supsetneq \m:\m$.

iii $\implies$ v: From $(A:\m)\m=A$ we see that we can write $1= \sum_{i=1}^m x_i y_i$ where $x_i \in A:\m$ and $y_i \in \m$. Pick $i$ such that $x_i y_i \in A^*$. We claim that $\m=(y_i)$. Indeed for $x \in \m$ we find
\begin{eqnarray*}
x=y_i \cdot \frac{x x_i}{y_i x_i} \in (y_i). 
\end{eqnarray*} 

v $\implies$ vi: We know that $\m$ is principal and that $A$ is local noetherian and has dimension $1$ (as it is integral over $R$). This makes $A$ into a regular local ring. By Corollary 10.14 from \cite{EI} it follows that $A$ is a domain. Now apply Proposition 9.2 from \cite{AT} to see that $A$ is a discrete valuation ring. 

vi $\implies$ i: Again apply Proposition 9.2 from \cite{AT} to see that $A$ is integrally closed. 
\end{proof}

\begin{theorem} \label{thm1} \label{maino}
Assume that $A$ is tame at $\pa$. Then $\pa A^{\dagger} \subseteq A$ if and only if $A=\overline{A}$. 
If $A \subsetneq \overline{A}$, we have $(\m:\m)/\m= A/\m \oplus \left((\m:\m) \cap \pa A^{\dagger}\right)/\m$. 
\end{theorem}
\begin{proof}
As taking traces behaves well with respect to tensoring, we obtain the following commutative diagram 
\[
\xymatrix{ 
Q(A) \ar[d]_{\mathrm{Tr}_{Q(A)/Q(R)}} & A \ar[rr] \ar[l] \ar[d]_{\mathrm{Tr}_{A/R}} & & A/\pa A  \ar[d]_{\mathrm{Tr}_{A/\pa A/R/\pa}} \\
Q(R) & R \ar[rr] \ar[l] &  & R/\pa.  } \]
Consider the symmetric bilinear form on $A/\pa A \times A/\pa A \to R/\pa$ obtained from $\mathrm{Tr}_{A/\pa A / R/\pa}$. By tameness the radical of
this form is $\m/\pa A$. Hence we obtain a non-degenerate form $A/\m \times A/\m \to R/\pa$. As this trace form is induced by the trace form
$\mathrm{Tr}_{Q(A)/Q(R)}$ and $\pa$ is principal, we see that 
\begin{eqnarray*}
\pa A^{\dagger} \cap A = \{\pa x: \mathrm{Tr}_{Q(A)/Q(R)}(xA) \subseteq R \} \cap A = \{y \in A: \mathrm{Tr}_{Q(A)/Q(R)}(yA) \subseteq \pa \} = \m.
\end{eqnarray*}

$\implies$: Suppose that $A$ is not integrally closed and let $T=\m:\m$. By Theorem \ref{wuffer} we have $T \supsetneq A$ and $T$ is an order in
$Q(A)$ by Lemma \ref{sop}. Hence $T$ comes with a trace form, which is induced from $\mathrm{Tr}_{Q(A)/Q(R)}$. Notice that $\m \subset T$ is an ideal.
Let $p: R \to R/\pa$ be the reduction. Then for $x \in \m$ we have $p \circ \mathrm{Tr}_{Q(A)/Q(R)}(xT)=0$. Now let $\varphi: T/\m \times T/\m\to
R/\pa$ be the map defined by $(t+\m,t'+\m) \mapsto p \circ \mathrm{Tr}_{Q(A)/Q(R)}(tt')$. Let $\psi: A/\m \times A/\m \to R/\pa$ be the map defined by
$(a+\m,a'+\m) \mapsto p \circ \mathrm{Tr}_{Q(A)/Q(R)}(aa')$. Then we have the following commutative diagram:
\[
\xymatrix{ 
T/\m \times T/\m \ar[r]^{\varphi}  & R/\pa \\
A/\m \times A/\m \ar[u]^{i \times i} \ar[r]^{\psi}  & R/\pa. \ar[u]_{\mathrm{id}} 
}   
\]
We know that $A/\m \subsetneq T/\m$ is non-degenerate. If denote by $^{\perp}$ the orthogonal complement, then by Proposition 1.7 from \cite{EL} we have
\begin{eqnarray*}
T/\m &=& A/\m \perp (A/\m)^{\perp} \\
&=& A/\m \perp \left( T \cap \pa A^{\dagger}\right) / \m,
\end{eqnarray*}
which proves the last statement. 
As $T/\m \supsetneq A/\m$, it follows that $\left( T \cap \pa A^{\dagger}\right) / \m \neq 0$. Suppose that $\pa A^{\dagger} \subseteq A$, then we
have
\begin{eqnarray*}
 (T \cap \pa A^{\dagger})/\m=(T \cap A \cap \pa A^{\dagger})/\m=(A \cap \pa A^{\dagger})/\m=\m/\m=0,
\end{eqnarray*}
a contradiction.

$\limplies$: By Theorem \ref{wuffer} we see that $A$ is a discrete valuation ring. First notice that $\pa A^{\dagger} \subseteq Q(A)$ is an $A$-module. Now suppose that $\pa A^{\dagger} \not \subseteq A$, then $A \subseteq \pa A^{\dagger}$ (here we use that $A$ is a discrete valuation ring). Hence we have $A \subseteq \pa A^{\dagger} \cap A=\m$, a contradiction.  
\end{proof}

\begin{example} 
Let $R=\Z$ and let $A$ be an order over $\Z$ which is tame at the prime $p$. Then the statement says that $A$ is integrally closed at $p$ if and only
if the finite group $A^{\dagger}/A$ has no element of order $p^2$. 
\end{example}

We have the following corollary.

\begin{corollary} \label{nutto}
Assume that $A$ is tame at $\pa$. Let $B=A^{\dagger}/A$. Then we have
\begin{eqnarray*}
(\m:\m)/A = (\pa B)[\m].
\end{eqnarray*}
If we have  $A \neq \overline{A}$, then we have
\begin{eqnarray*}
(\m:\m)/A = (\pa B)[\m]=B[\m],
\end{eqnarray*}
where for any $A$-module $A'$ we set $A'[\m]=\{x \in A': \m x=0 \}$.
\end{corollary}
\begin{proof}
We will first prove the statement if $A=\overline{A}$. Then $\m:\m=A$ (Theorem \ref{wuffer}) and $\pa B =0$ (Theorem \ref{thm1}). The statement in this case now follows directly.

Now suppose that $A \neq \overline{A}$. From Theorem \ref{wuffer} it follows that $A:\m=\m:\m$. As $\m:\m \subseteq \overline{A} \subseteq A^{\dagger}$, it follows that $(\m:\m)/A=(A:\m)/A=B[\m]$. Now we obviously have $(\pa B)[\m] \subseteq B[\m]$. From Theorem \ref{thm1} it follows that $(\m:\m)/\m=A/\m+\left((\m:\m) \cap \pa A^{\dagger} \right)/\m$. If we now take the quotient by $A/\m$ we obtain
\begin{eqnarray*}
(\m:\m)/A = \left(A+(\m:\m) \cap \pa A^{\dagger} \right)/A.
\end{eqnarray*}
This shows that $(\m:\m)/A \subseteq (\pa B)[\m]$ and this finishes our proof. 
\end{proof}

\section{The connection between anisotropy and the integral closure}

In this section let $R$ be a discrete valuation ring with maximal ideal $\pa$ and let $A$ be an order over $R$. If $I$ is a nonzero ideal of $R$,
then one easily finds $I^{-1}/R \cong_R R/I$. 

\begin{lemma} \label{domino}
Let $I=\mathrm{Ann}_R(A^{\dagger}/A)$. Then we have the following non-degenerate symmetric $R/I$-bilinear form:
\begin{eqnarray*}
\langle\,\, ,\,\rangle: A^{\dagger}/A \times A^{\dagger}/A &\to& I^{-1}/R \\
(x+A,y+A) &\mapsto& \mathrm{Tr}_{Q(A)/Q(R)}(xy)+R.
\end{eqnarray*}
\end{lemma}
\begin{proof}
One easily sees that this map is well-defined. We will give a sketch of the rest of the proof, see \cite{KO1} Lemma 4.1.3 for the details. As $Q(A)$
is a finite \'etale $Q(R)$-algebra, 
it follows that the natural map $A^{\dagger} \to \mathrm{Hom}_R(A,R)$ is an isomorphism. One can use this to show that $A=A^ {\dagger \dagger}$. The
non-degeneracy then follows from this and the fact that $\mathrm{length}_R(A^{\dagger}/A)=\mathrm{length}_R(\mathrm{Hom}_R(A^{\dagger}/A,I^{-1}/R))$.
\end{proof}

\begin{lemma} \label{domina}
Let the notation be as in Lemma \ref{domino}. 
Suppose that $\overline{A}$ is tame at $\pa$. Then $C=\overline{A}/A \subseteq A^{\dagger}/A$ satisfies $\pa C^{\perp} \subseteq C \subseteq
C^{\perp}$.
\end{lemma}
\begin{proof}
A simple calculation shows that $C^{\perp}= \overline{A}^{\dagger}/A$. As
$\overline{A} \subseteq \overline{A}^{\dagger}$ by Lemma \ref{longi} we have $C
\subseteq C^{\perp}$. The tameness assumption on $\overline{A}$ implies by
Theorem \ref{maino} that $\pa \overline{A}^{\dagger} \subseteq \overline{A}$ and
hence that $\pa \left( R/I \right) \overline{A}^{\dagger}/A \subseteq
\overline{A}/A$ and hence $\pa \left( R/I \right) C^{\perp} \subseteq C$. 
\end{proof}

Notice that $R/I$ from the previous lemma is an artinian principal
ideal ring. 
This lemma forms the connection between the integral closure and anisotropy. We
recall some definitions from \cite{KO2} first. Let $(R',\m)$ be an artinian
local principal ideal ring and let $n$ be its length. Let $M$ be a finitely
generated $R'$-module. Let $N$ be an $R'$-module such that $N \cong_{R'} R'$ and
let $\bi: M \times M \to N$ be a non-degenerate symmetric $R'$-bilinear form.
The \emph{radical root} of $(M,\bi$) is now defined as
\[\mathrm{rr}(M)=\bigcap_{L \subseteq M:\ \m L^{\perp} \subseteq L \subseteq
L^{\perp}} L,\] where all $L$ are $R'$-modules. We define the \emph{lower root} of $M$ as follows:
\begin{eqnarray*}
\mathrm{lr}(M)= \sum_{k=0}^n \left( \m^k M \cap M[\m^k] \right),
\end{eqnarray*}
where $M[\m^k]=\{ x \in M: \m^k x=0 \}$. 
The form $\bi$ is called \emph{anisotropic} if the lower root of $M$ is the unique submodule $L$ of $M$ satisfying $\m L^{\perp} \subseteq L \subseteq L^{\perp}$. 
We remark that in \cite{KO2} it is shown how to calculate $\mathrm{lr}(M)$ and check if a form is anisotropic. In \cite{KO2} a formula is given for
$\mathrm{rr}(M)$ if $\mathrm{char}(R/\m) \neq 2$. We have the following lemma.

\begin{lemma} \label{or}
Assume that $\bi: M \times M \to N$ is non-degenerate. The following statements hold.
\begin{enumerate}
 \item 
  Suppose that $M$ is cyclic. Then $\bi$ is anisotropic.
 \item
 Suppose that $M$ is generated by two elements and $\mathrm{length}_{R'}(M)$ is odd. Then $\bi$ is anisotropic. 
\end{enumerate}
\end{lemma}
\begin{proof}
 See \cite{KO2} Remark 5.3. 
\end{proof}

We can now give the connection between anisotropy and the integral closure.

\begin{theorem} \label{main5}
Suppose that $\overline{A}$ is tame at $\pa$. Let $B= A^{\dagger}/A$ and let
$I=\mathrm{Ann}_R(A^{\dagger}/A)$. Consider the form $\langle\,\, ,\,\rangle$
from Lemma \ref{domino}. Let $D \subset A^{\dagger}$ be such that $D/A=\mathrm{rr}(B)$.
Let $A[D]$ be the smallest ring inside $Q(A)$ containing $A$ and $D$. Then the
following statements hold.
\begin{enumerate}
\item
We have $\mathrm{rr}(B)=D/A \subseteq A[D]/A \subseteq \overline{A}/A $.
\item
Suppose that $\langle\,\, ,\,\rangle$ is anisotropic. Then $\overline{A}/A=\mathrm{lr}(B)$.  
\item
Suppose that $\mathrm{rr}(B)$ satisfies $\pa \left(R/I \right) \cdot
\mathrm{rr}(B)^{\perp} \subseteq \mathrm{rr}(B)$. Assume that $A[D]$ is tame at
$\pa$. Then $\overline{A}/A=A[D]/A$.
\end{enumerate}
\end{theorem}
\begin{proof} 
i. We have $\mathrm{rr}(B)=D/A \subseteq \overline{A}/A$ by Lemma \ref{domina}.
As $\overline{A}$ is a ring, it follows that $A[D] \subseteq \overline{A}$. 

\noindent ii. We directly obtain the result by definition of anisotropy and Lemma \ref{domina}. 

\noindent iii. We know that $A[D]/A \subseteq \overline{A}/A$ by
i. Notice that
\begin{eqnarray*}
\mathrm{rr}(B) \subseteq A[D]/A \subseteq \left( A[D]/A \right)^{\perp} \subseteq \mathrm{rr}(B)^{\perp}. 
\end{eqnarray*}
Hence $\pa \left( A[D]/A \right)^{\perp} \subseteq A[D]/A$.
As $A[D]$ is an order which is tame at $\pa$ and $(A[D]/A)^{\perp}=A[D]^{\dagger}/A$, we can apply Theorem \ref{maino}
to see that $\overline{A}/A=A[D]/A$. 
\end{proof}

Later, in Theorem \ref{nano}, we will see that under certain hypotheses we have the surprising equality $A[D]=D$. 

\section{A sufficient condition for tameness} \label{ta}

In this section we will prove a condition which implies tameness and is easy to check. Recall that $R$ is a discrete valuation ring with prime
$\pa=(\pi)$ and $A$ is an order over $R$.

\begin{theorem} \label{tammy2}
Let $B=A^{\dagger}/A$. Let $A'$ be an $R$-order with $A \subseteq A' \subseteq \overline{A}$. Then $A'$ is tame at $\pa$ if for all maximal ideals $\m
\subset A$ we have $\mathrm{dim}_{R/\pa}(B_{\m}/\pa B_{\m})<\mathrm{char}(R/\pa)$ or $\mathrm{char}(R/\pa)=0$. Furthermore, the dimensions of $B/\pa
B$ and the trace radical of $A/\pa A$ over $R/\pa$ are equal.
\end{theorem}
\begin{proof}

After tensoring with $\hat{R_{\pa}}$, we may assume that $R$ is a complete discrete valuation ring and that $A=\prod_{\m \in
\mathrm{Maxspec}(A)} A_{\m}$ and $A'= \prod_{\m \in \mathrm{Maxspec}(A)} A'_{\m}$ (Lemma \ref{mars}). Let $B'=A'^{\dagger}/A'$ and let $\m' \in
\mathrm{Maxspec}(A')$ lying above $\m=A \cap \m'$. By Lemma \ref{sor} we have $B_{\m}=A_{\m}^{\dagger}/A_{\m}$ and
$B'_{\m'}=A_{\m'}^{\prime \dagger}/A'_{\m'} \subseteq A_{\m}^{\prime \dagger}/A'_{\m}$ (Lemma \ref{mars}). We have a natural injective map 
\begin{eqnarray*}
A_{\m}^{\prime \dagger}/A'_{\m} \to A_{\m}^{\dagger}/A'_{\m} \cong \left( A_{\m}^{\dagger}/A_{\m} \right)/ \left( A_{\m}'/A_{\m} \right),
\end{eqnarray*}
which shows that $B'_{\m'}$ is a
quotient of a submodule of $B_{\m}$ and this shows that $\mathrm{dim}_{R/\pa}(B'_{\m'}/\pa B'_{\m'}) \leq \mathrm{dim}_{R/\pa}(B_{\m}/\pa B_{\m})$.
Hence we can assume
that $A=A'$. 

If $\mathrm{char}(R/\pa)=0$, then $A/\pa A$ will be automatically tame. 

Assume that $\mathrm{char}(R/\pa) \neq 0$. Assume that $A/\pa A$ is wild. We have that \[A/\pa A=\bigoplus_{\m \supseteq \pa A} \left(A/\pa A
\right)_{\m}=\bigoplus_{\m \supset \pa A} A_{\m}/\pa A_{\m}\] (exactness of localization and Theorem 2.13 from \cite{EI}). It follows
that there is a prime $\m$ such that $A_{\m}/\pa A_{\m}$ is wild. By Lemma \ref{sor} we may assume that $A$ is local. Let $C=A/\pa A$, which
we assume to be wild over $R/\pa$. Then it follows that
$\mathrm{dim}_{R/\pa}(C^{\perp}) \geq
\mathrm{char}(R/\pa)$ (Proposition \ref{weed}). For $x \in A$ we have $x+\pa A \in C^{\perp}$ iff
$\mathrm{Tr}_{Q(A)/Q(R)}(xA) \subseteq (\pi)$ iff $\mathrm{Tr}_{Q(A)/Q(R)}(\frac{x}{\pi}A) \subseteq A$ iff $\frac{x}{\pi} \in A^{\dagger}$ iff $x \in
\pi A^{\dagger}=\pa A^{\dagger}$.  Hence
\begin{eqnarray*}
C^{\perp} = \left( \pa A^{\dagger} \cap A \right) / \pa A.  
\end{eqnarray*} 
We have
\begin{eqnarray*}
\left( \pa A^{\dagger} \cap A \right)/\pa A &=& \left( \pi A^{\dagger} \cap A \right)/\pi A \\
&\cong & \left(A^{\dagger} \cap \pi^{-1} A \right)/A \\
&=& B[\pa].
\end{eqnarray*}
Finally consider the following exact sequence:
\[
\xymatrix{ 
0 \ar[r]^{} & B[\pa] \ar[r] & B \ar[r]^{\cdot \pi} & B \ar[r] & B/\pa B \ar[r]& 0.
}
\] 
The length as $R$-module is an additive function (\cite{AT}, Proposition 6.9). Hence $\mathrm{length}_R(B[\pa])=\mathrm{length}_R(B/\pa B)$, and both
lengths are their dimensions over $R/\pa$. So if $A$ is wild at $\pa$ we have 
\begin{eqnarray*}
\mathrm{dim}_{R/\pa}(B/\pa B) &=& \mathrm{dim}_{R/\pa}(B[\pa]) \\
&=& \mathrm{dim}_{R/\pa}(C^{\perp}) \\
&\geq& \mathrm{char}(R/\pa),  
\end{eqnarray*}
and this concludes the proof.
\end{proof}

\begin{remark}
As $B/\pa B= \bigoplus_{\m \supset \pa A} \left( B/\pa B \right)_{\m}=\bigoplus_{\m \supset \pa A} B_{\m}/\pa B_{\m}$, the condition in
the above theorem is satisfied if $\mathrm{dim}_{R/\pa}(B/\pa B)<\mathrm{char}(R/\pa)$ 
\end{remark}

We can finally prove Theorem \ref{1} from the introduction.

\begin{proof}[Proof of Theorem \ref{1}.]
Use Lemma \ref{floop} to reduce to the case where we work over the localization of $\Z$ at $p$. Now use Theorem \ref{maino} in combination with
Theorem \ref{tammy2}. Here we remark that for $B=A^{\dagger}/A$ we have $\mathrm{dim}_{\Z/p\Z}(B/pB) \leq [K:\Q]$.
\end{proof}

\begin{theorem} \label{zoss}
 Let $B=A^{\dagger}/A$ and suppose that $2 \neq \mathrm{char}(R/\pa)$. Suppose that one of the following conditions is satisfied:
\begin{enumerate}
 \item 
 $B$ is cyclic as an $R$-module;
 \item
 $B$ is generated as an $R$-module by two elements and $\mathrm{length}_R(B)$ is odd.
\end{enumerate}
Then $\overline{A}/A=\mathrm{lr}(A^{\dagger}/A)$. 
\end{theorem}
\begin{proof}
 Theorem \ref{tammy2} shows that we are in a tame case. Now combine Lemma \ref{or} and Theorem \ref{main5}.
\end{proof}

One can show that case i in the above lemma can never happen if $\mathrm{char}(R/\pa)=2$ (see \cite{KO1} Lemma 5.4.2). 

We can now also prove Theorem \ref{2}.

\begin{proof}[Proof of Theorem \ref{2}.]
 Use Lemma \ref{floop} to reduce to the local case. Now use Theorem \ref{zoss} to finish the proof. 
\end{proof}

\begin{example}
 Theorem \ref{2} is false if $2|\#G$. Let $A=\Z[\sqrt{5}]$. Then we have $A^{\dagger}/A \cong \Z/2\Z \times \Z/2 \cdot 5 \Z$, but
$\overline{A}=\Z[\frac{1+\sqrt{5}}{2}] \supsetneq A$. 
\end{example}

\section{Galois orders}

In this section let $R$ be a Dedekind domain (not necessarily a discrete valuation ring) and let $A$ be an order over $R$. We will present
another condition for tameness in the case that a group $G$ acts in a nice way on $A$ (as will be explained later).

\begin{definition}
Let $S$ be a nonzero $K$-algebra where $K$ is a field and let $G$ be a group acting on $A$ through $K$-automorphisms. Then $S$ is called a
\emph{finite Galois algebra} over $K$ if $S$ is a finite \'etale $K$-algebra, $\# G=\mathrm{dim}_K(S)$ and $S^G=K$. 
\end{definition}

\begin{remark}
 There are many other equivalent definitions of finite Galois algebras. One of the statements is that $S$ is a Galois algebra with group $G$ if
and only if $S$ is isomorphic as a $K$-algebra with $G$-action to $_H\mathrm{Map}(G,L)$ where $L/K$ is a Galois extension with group $H$ together with
an embedding $H \to G$.
\end{remark}

\begin{remark} \label{5}
Let $S$ be a finite Galois algebra over $K$ with group $G$. Let $K \to K'$ be a morphism of fields. Then $S \otimes_K K'$ is still a
finite Galois algebra over $K'$ with group $G$. 
\end{remark}

\begin{definition}
Let $G$ be a finite group acting on $A$ by $R$-algebra automorphisms. Then $G$ acts naturally on $Q(A)=A \otimes_R Q(R)$ by $Q(R)$-algebra
automorphisms. We call $A$ a \emph{Galois order} over $R$ with group $G$ if $Q(A)$ together with $G$ is a finite Galois algebra over $Q(R)$. Remark
that in such a case we have $A^G=Q(A)^G \cap A=Q(R) \cap A=R$.
\end{definition}

\begin{example}
 Let $K$ be a number field which is Galois over $\Q$ with group $G$. Then any order $A$ stable under $G$ is a Galois order with group $\{g|_A: g \in
G\}$. 
\end{example}

For a prime $\qa \subset A$ lying over $\pa \subset R$ 

\begin{definition}
Let $\qa \subset A$ be a prime lying over $\pa \subset R$. We define define the \emph{decomposition group} of $\qa$ over $\pa$ to be $G_{\qa/\pa}=\{g
\in G: g(\qa)=\qa \}$. Consider the natural map $\varphi: G_{\qa/\pa} \to \mathrm{Aut}_{R/\pa}(A/\qa)$. Then we define the
\emph{inertia group} of
$\qa$ over $\pa$ to be $I_{\qa/\pa}=\mathrm{ker}(\varphi) \subseteq G_{\qa/\pa}$.  
\end{definition}

\begin{lemma} \label{sos}
Let $B$ be a commutative ring and $G \subseteq \mathrm{Aut}(B)$ a finite group. If $\varphi, \psi: B \to k$ are ring momorphisms to a domain $k$ that
coincide on $B^G$, then $\varphi=\psi \circ g$ for some $g \in G$.
\end{lemma}
\begin{proof}
 See \cite{ST}, Lemma 15.1 for an elegant proof. 
\end{proof}

\begin{lemma} \label{43}
Let $A$ be a Galois order with group $G$ over $R$. Let $\pa \subset R$ be prime and let $\qa \subset A$ be a prime lying above $R$. Then the following
statements hold:
\begin{enumerate}
\item
The group $G$ acts transitively on the set of primes of $A$ lying above $\pa$. 
\item
The map $\varphi: G_{\qa/\pa} \to \mathrm{Aut}_{R/\pa}(A/\qa)$ is surjective.
\item
The extension $A/\qa$ over $R/\pa$ is normal. 
\end{enumerate}
\end{lemma}
\begin{proof}
For the first two parts we give a sketch since this is well-known. 

i. For two primes $\qa_1, \qa_2 \subset A$ above $\pa$ we consider two maps $A \to
A/\qa_i \to \overline{Q(R/\pa)}$ (algebraic closure of $Q(R/\pa)$) and apply Lemma \ref{sos}. 

ii. For the second part, we consider maps $A \to A/\qa \overset{f}{\to} A/\qa$ where $f \in \mathrm{Aut}_{R/\pa}(A/\qa)$. Apply
Lemma \ref{sos} to see that there is an element $g \in G$ that maps to $f$.

iii. Take $\overline{a} \in A/\qa$ where $a \in A$. Then $\prod_{g \in G}(X-\overline{g}(a))=\overline{\prod_{g \in G}(X-g(a))} \in R/\pa[X]$, which
follows from the fact that $A^G=R$ as $A$ is a Galois algebra.
\end{proof}

The concepts defined above behave  well under localization and completion. 

\begin{lemma} \label{gas}
 Suppose that $A$ is a Galois order over $R$ with group $G$. Let $\pa \subset R$ be prime and let $\qa \subset A$ be a prime lying over $\pa$. Then
the following statements hold.
\begin{enumerate}
 \item 
 $A_{\pa}=A \otimes_R R_{\pa}$ is a Galois order over $R_{\pa}$ with group $G$.
 \item
 $A \otimes_R \hat{R_{\pa}}$ is a Galois order over $\hat{R_{\pa}}$ with group $G$.  
 \item
 We have $G_{\qa/\pa}=G_{\qa A_{\pa}/\pa R_{\pa}}$ and $I_{\qa A_{\pa}/\pa R_{\pa}}=I_{\qa/\pa}$.
 \item
 Write $A \otimes_R \hat{R_{\pa}}=\prod_{\m \supseteq \pa A} \hat{A_{\m}}$. Then
$G_{\qa/\pa}=G_{\qa \hat{A_{\qa}}/\pa \hat{R_{\pa}}}$ and $I_{\qa/\pa}=I_{\qa \hat{A_{\qa}}/\pa \hat{R_{\pa}}}$. 
\item
$\hat{A_{\qa}}$ is a Galois order with group $G_{\qa/\pa}$ over $\hat{R_{\pa}}$. 
\end{enumerate}
\end{lemma}
\begin{proof}
i. This is obvious since we still have the same total quotient ring. 

ii. Notice that 
\begin{eqnarray*}
Q(A \otimes_R \hat{R})= A \otimes_R \hat{R} \otimes_{\hat{R}} Q(\hat{R}) =A \otimes_R Q(\hat{R}) = \left(A \otimes_R Q(R)\right) \otimes_{Q(R)}
Q(\hat{R}).
\end{eqnarray*}
Now use the fact that Galois algebras behave well with respect to base change (Remark \ref{5})

iii. One can easily check this.

iv. Use the proof of Lemma \ref{cheata} to see that the elements of $G_{\qa/\pa}$ correspond exactly to the elements which map $\hat{A_{\qa}}$ to
itself.
We have $\hat{A_{\qa}}/\qa \hat{A_{\qa}}=A/\qa$ and the natural map $G_{\qa \hat{A_{\qa}}/\pa \hat{R}} \to \mathrm{Aut}(\hat{A_{\qa}}/\qa
\hat{A_{\qa}})$ still has kernel $I_{\qa/\pa}$. 

v. First of all, we have seen that $\hat{A_{\qa}}$ is an order (Lemma \ref{cheata}). Using the decomposition $A \otimes_R \hat{R_{\pa}}=\prod_{\m}
\hat{A_{\m}}$ and the fact that $G$ acts transitively on the set of primes (see Lemma \ref{43}i), we see that $\#
G_{\qa/\pa}=\mathrm{dim}_{Q(\hat{R_{\pa})}}(Q(\hat{A_{\qa}}))$ as
required. Suppose that $a \in \hat{A_{\qa}}$ is fixed by all elements of $G_{\qa/\pa}$. For $\m \in \mathrm{Maxspec}(A)$ let $g_{\m} \in G$ be an
element such that $g_{\m}$ maps
$\hat{A_{\qa}}$ to $\hat{A_{\m}}$. We pick $g_{\qa}=\mathrm{id} \in G$. Then consider $(g_{\m}(a))_{\m} \in \prod_{\m} \hat{A_{\m}}=A \otimes_R
\hat{R_{\pa}}$. We claim that this element
is fixed by $G$. Indeed, if $g \in G$ maps $\m$ to $\m'$, then $g_{\m'}^{-1}gg_{\m} \in G_{\qa/\pa}$ and hence $g_{\m'}^{-1}gg_{\m}(a)=a$ and
$gg_{\m}(a)=g_{\m'}(a)$ as required.
As $A \otimes_R \hat{R_{\pa}}$ is a Galois order over $\hat{R_{\pa}}$, we conclude that $a \in \hat{R_{\pa}}$ as required. Hence the
statement follows.
\end{proof}

\begin{theorem} \label{zoe}
 Let $A$ be a Galois order over $R$ with group $G$. Let $\qa$ be a prime of $A$ and let $\pa=\qa \cap R$. Then $A$ is tame at $\pa$ if
and only if $\mathrm{char}(R/\pa) \nmid \# I_{\qa/\pa}$. 
\end{theorem}
\begin{proof}
As $G$ acts transitively on the primes lying above $R$ (Lemma \ref{43}), $A$ is tame at $\pa$ iff $A/\qa$ is a tame $R/\pa$-algebra. By Lemma
\ref{43} it follows that the map $G_{\qa/\pa}/I_{\qa/\pa} \to \mathrm{Aut}_{R/\pa}(A/\qa)$ is surjective. This lemma also gives us that the extension
$A/\qa$ is normal over $R/\pa$ and hence that $\#G_{\qa/\pa}/I_{\qa/\pa}=\#
\mathrm{Aut}_{R/\pa}(A/\qa)=[A/\qa:R/\pa]_s$, the separability degree of the extension. Let $i$ be the inseparability degree of this extension.
Notice that we have $\mathrm{dim}_{Q(R)}(Q(A))=\mathrm{dim}_{R/\pa}(A/\pa A)$. Indeed, both are equal to $\mathrm{rank}_{R_{\pa}}(A \otimes_R
R_{\pa})$. Then we have
\begin{eqnarray*}
 \#G &=& \mathrm{dim}_{Q(R)}(Q(A))=\mathrm{dim}_{R/\pa}(A/\pa A) \\
&=& \sum_{\qa' \supset \pa} \mathrm{dim}_{R/\pa}\left((A/\pa A)_{\qa'}\right) \\
&=& \#G/G_{\qa/\pa} \cdot \mathrm{dim}_{R/\pa}\left((A/\pa A)_{\qa}\right) \hspace{4.4cm} (\mathrm{Lemma\ } \ref{43}) \\
&=& \#G/G_{\qa/\pa} \cdot \mathrm{dim}_{R/\pa}(A/\qa) \cdot \mathrm{length}_{(A/\pa A)_{\qa}}((A/\pa A)_{\qa}) \\
&=& \#G/G_{\qa/\pa} \cdot i \cdot [A/\pa:R/\pa]_s \cdot e_{\qa/\pa} \\
&=& \#G/G_{\qa/\pa} \cdot i \cdot \#G_{\qa/\pa}/I_{\qa/\pa} \cdot e_{\qa/\pa}.
\end{eqnarray*}
Hence $\#I_{\qa/\pa}=i \cdot e_{\qa/\pa}$. As $i$ is always a power of $\mathrm{char}(R/\pa)$, the definition of tameness of $A$ at $\qa$ is
equivalent to saying that
$\mathrm{char}(R/\pa) \nmid i \cdot e_{\qa/\pa}$, that is, $\mathrm{char}(R/\pa) \nmid \#I_{\qa/\pa}$. 
\end{proof}

\section{Quasi-anisotropy and the integral closure}

First we will recall the definition of quasi-anistropy from \cite{KO2}. Let $M$ be a finitely generated module over an artinian local
principal ideal ring $(R',\pa')$. Let $N$ be an $R'$-module such that $N \cong_{R'} R'$ and let $\bi: M \times M \to N$ be a non-degenerate symmetric
$R'$-bilinear form. Then $\bi$ is called quasi-anisotropic if for all $R'$-submodules $L \subseteq \mathrm{lr}(M)$ we have
$\mathrm{lr}(L^{\perp}/L)=\mathrm{lr}(M)/L$. In this case we have $\mathrm{rr}(M)=\mathrm{lr}(M)$ (see \cite{KO2}, Lemma 10.8). In \cite{KO2} some
other
equivalent definitions of quasi-anisotropy are given which are more practical. The following lemma gives the connection between quasi-anisotropy
and anisotropy (see \cite{KO2}, Theorem 9.4).

\begin{lemma} \label{bra}
 Let $\bi: M \times M \to N$ be a non-degenerate symmetric $R'$-bilinear form. Then $\bi$ is quasi-anisotropic if and only if the induced form
$\bi':M/M[\pa'] \times M/M[\pa'] \to N/N[\pa']$ is anisotropic. 
\end{lemma}

We also have the following lemma (Lemma 9.5 from \cite{KO2}). 

\begin{lemma} \label{brak}
  Let $\bi: M \times M \to N$ be a non-degenerate symmetric $R'$-bilinear form which is quasi-anisotropic. Let $L \subseteq \mathrm{lr}(M)$. Then $L
\subseteq L^{\perp}$ and
the induced form $\bi': L^{\perp}/L \times L^{\perp}/L \to N$ is also quasi-anisotropic. 
\end{lemma}

For any ring $B$ we define the Jacobson radical $\mathfrak{r}_B$ to be the intersection of all maximal ideals of $B$. 

\begin{lemma} \label{cos}
Let $R$ be a complete discrete valuation ring and let $A$ be an order over $R$. Then $\mathfrak{r}_A:\mathfrak{r}_A$ is an order and $A$ is integrally
closed iff $\mathfrak{r}_A:\mathfrak{r}_A=A$.
\end{lemma}
\begin{proof}
 Write $A= \prod_{\m} A_{\m}$ as in Theorem \ref{ks}, where the $A_{\m}$ are local with maximal ideal $\m A_{\m}$. Now from Lemma \ref{sop} we know
that
$\mathfrak{r}_A:\mathfrak{r}_A=\prod_{\m} \left(\m A_{\m}:\m A_{\m} \right)$ is an order. We know that $\overline{A}=\prod_{\m} \overline{A_{\m}}$,
and hence $A$
is integrally closed iff all $A_{\m}$ are integrally closed. Here one uses that total quotient ring is just the product of the total
corresonding total quotient rings. By Theorem \ref{wuffer} we know that $A_{\m}$ is integrally closed iff $\m A_{\m}: \m A_{\m}=A_{\m}$. Hence we see
that $A$ is integrally closed iff $A=\mathfrak{r}_A:\mathfrak{r}_A$. 
\end{proof}

We have the following theorem. The hard part is to prove that a certain module is in fact already a ring. 

\begin{theorem} \label{nano}
Let $A$ be an order over a discrete valuation ring $(R,\pa)$. Let $I \subseteq \mathrm{Ann}_R(A^{\dagger}/A)$ be a nonzero ideal of $R$. Let $A_0=A
\otimes_R \hat{R}$
and
$A_{i+1}=\mathfrak{r}_{A_i}:\mathfrak{r}_{A_i}$ for $i \geq 0$. Suppose that the $A_i$ are tame at $\pa \hat{R}$. Let $B= A^{\dagger}/A$ and let
$\langle\,\,
,\,\rangle: B \times B \to I^{-1}/R$ be the induced form (Lemma \ref{domino}). Suppose that $\langle\,\, ,\,\rangle$ is
quasi-anisotropic. Then $\overline{A}/A=\mathrm{lr}(B)$. 
\end{theorem}
\begin{proof}
Assume that $R$ is complete (Lemma \ref{cheat}). We will give a proof by induction
on $s=\mathrm{length}_R(\mathrm{lr}(B))$. If $s=0$ then $\pa B=0$ and by Theorem \ref{thm1} we find $\overline{A}/A=A/A=\mathrm{lr}(B)$. Now continue
by induction and assume that $s \geq 1$. We know that $\mathrm{lr}(B)=\mathrm{rr}(B) \subseteq \overline{A}/A \subseteq A^{\dagger}/A$ (Theorem
\ref{main5} and the quasi-anisotropy). As $s \geq 1$, this implies that $A \subsetneq \overline{A}$. Now consider the order
$A'=\mathfrak{r}_A:\mathfrak{r}_A$. This order $A'$ satisfies $A \subsetneq A' \subseteq \overline{A} \subseteq A^{\dagger}$ (Lemma \ref{cos}). By
using Corollary \ref{nutto} and Theorem \ref{ks} we have $A'/ A \subseteq \pa B \cap B[\pa] \subseteq \mathrm{lr}(B)$ (by definition of the lower
root). Now use Lemma
\ref{brak} to see that $\left( A'^{\dagger}/A \right)/\left( A'/A \right) \cong A'^{\dagger}/A'$ is still quasi-anisotropic. We have by the definition of
quasi-anistropy that
$\mathrm{lr}(\left( A'^{\dagger}/A \right)/\left( A'/A\right))=\mathrm{lr}(B)/\left( A'/A \right)$,  which has smaller length than
$\mathrm{lr}(B)$ (as $A \subsetneq A'$). By our induction hypothesis we have 
\begin{eqnarray*}
\mathrm{lr}(B)/\left( A'/A \right) &=& \mathrm{lr}(\left( A'^{\dagger}/A \right)/\left( A'/A \right)) \cong \mathrm{lr}(A'^{\dagger} /A' ) \\
&=& \overline{A'}/A' = \overline{A}/A' \cong \left( \overline{A}/A \right)/ \left( A'/A \right).
\end{eqnarray*}
As our maps are natural, this gives $\mathrm{lr}(B)=\overline{A}/A$ and hence we are done.

In the proof we used tameness for $\overline{A}$ (which is one of the $A_i$) and the $A_i$.
\end{proof}

We now want some condition guaranteeing this tameness. 

\begin{lemma}
Let $A$ be an order over a discrete valuation ring $(R,\pa)$. Let $B=A^{\dagger}/A$. Let $A_0=A \otimes_R \hat{R}$ and
$A_{i+1}=\mathfrak{r}_{A_i}:\mathfrak{r}_{A_i}$ for $i \geq 0$. Then the orders $A_i$ are tame at $\pa \hat{R}$ if one of the following conditions is
satisfied.
\begin{enumerate}
\item
For every $\m \subset A$ maximal we have $\mathrm{dim}_{R/\pa}(B_{\m}/\pa B_{\m}) <\mathrm{char}(R/\pa)$;  
\item
We have $\mathrm{dim}_{R/\pa}(B/\pa B) <\mathrm{char}(R/\pa)$;
\item
$A$ is a Galois order over $R$ with group $G$ and for some prime $\m \subset A$ we have $\mathrm{char}(R/\pa) \nmid \# I_{\m/\pa}$;
\item
$A$ is a Galois order over $R$ with group $G$ and $\mathrm{char}(R/\pa) \nmid \# G$.
\end{enumerate}
\end{lemma}
\begin{proof}
i. We know that $B_{\m}=\hat{A_{\m}}^{\dagger}/\hat{A_{\m}}$ by Lemma \ref{sor}. We have $A \otimes_R \hat{R}= \prod_{\m} \hat{A_{\m}}$ (Lemma
\ref{cheata}). By Theorem \ref{tammy2} we know that all orders between
$\hat{A_{\m}}$ and $\overline{\hat{A_{\m}}}$ are tame at $\pa \hat{R}$. Hence all orders between $A \otimes_R \hat{R}$ and $\overline{A}
\otimes_R \hat{R}$ are tame (Lemma \ref{mars}).

ii. This condition implies the first condition. 

iii. If the assumption holds for a single $\m$, it holds for all primes above $\pa$, since the inertia groups are conjugate.

 By Lemma \ref{gas} we know that $A \otimes_R \hat{R}$ is a Galois order over $\hat{R}$ with group $G$, and its inertia groups are the $I_{\m/\pa}$
where $\m$ ranges over the primes of $A$ lying above $\pa$. 
We now claim that the $A_i$ are Galois orders with group $G$ over $\hat{R}$. Indeed, as they have the same quotient field as $A_0$, we just need to
check that $G$ acts on them. By induction, it is enough to check it for $A_1$. First notice that $\mathfrak{r}_{A_0}$ is stable under $G$. Let $x \in
A_1$ and $g \in G$. Then we have 
\begin{eqnarray*}
g(x) \mathfrak{r}_{A_0} = g(x \mathfrak{r}_{A_0}) \subseteq g(\mathfrak{r}_{A_0})=\mathfrak{r}_{A_0}. 
\end{eqnarray*}
So $g(x) \in A_1$ and we are done. 

We can write $A \otimes_R \hat{R} =
\prod_{\m} \hat{A_{\m}}$ and the elements of $G$ fixing the prime corresponding to $\hat{A_{\m}}$ are exactly those who fix $\hat{A_{\m}}$. But then
it follows that the inertia groups of the $A_i$ are subgroups of the $I_{\m/\pa}$. Hence $\mathrm{char}(R/\pa)$ doesn't divide the order of
any inertia group
occurring. By Theorem \ref{zoe} we see that all $A_i$ are tame at $\pa \hat{R}$ as required.

iv. As $I_{\m/\pa} \subseteq G_{\qa/\pa} \subseteq G$ as subgroups, we have $\# I_{\m/\pa} | \#G$ and the result follows from iii. 
\end{proof}

\section{Examples}

\begin{example}
Let $f=x^4-20x^3-20x^2+17x +2 \in \Z[x]$ and let $A=\Z[x]/(f(x))$. Then for the discriminant we have $\Delta(f)=7^4 \cdot 13 \cdot 11897$ and
$A^{\dagger}/A \cong \Z/7\Z \times \Z/(7^3
\cdot 13 \cdot 11897)\Z$.
There is only one prime $p$ which might satisfy $p|[\overline{A}:A]$, namely $7$. By Theorem \ref{maino} we have that $7|[\overline{A}:A]$. By Lemma
\ref{bra} and Lemma \ref{or} we see that the form corresponding to the prime $7$ is quasi-anisotropic. By Theorem \ref{nano} we have
\begin{eqnarray*}
\overline{A}/A &=& \mathrm{lr}(A^{\dagger}/A)=\left( A+ \frac{3\alpha^3+\alpha^2+2}{7} \Z \right)/A.
\end{eqnarray*}
\end{example}

\begin{example}
Let $f=x^4-625x^3-125x^2-15625x-15625 \in \Z[x]$ and let $A=\Z[x]/(f(x))$. Then $\Delta(f)= 5^{20} \cdot 13 \cdot 457 \cdot 8111$. In this case
$A^{\dagger}/A \cong \Z/5^3 \Z \times
\Z/5^7 \Z \times \Z/5^{10} \cdot 13 \cdot 457 \cdot 8111 \Z$. 
A calculation, which uses the algorithmic description of anisotropy from \cite{KO2}, shows that the form at the prime $5$ is anisotropic. Now use
Theorem \ref{main5} to get
\begin{eqnarray*}
\overline{A}= \mathrm{lr}(A^{\dagger}/A)=\Z+ \left( \frac{3}{3125} \alpha^3 +\frac{1}{25} \alpha \right) \Z+ \frac{1}{125} \alpha^2 \Z+ \frac{1}{625}
\alpha^3 \Z.
\end{eqnarray*}
\end{example}

\begin{remark}
 The examples above come from algebraic number theory. One can also make examples using for example function fields. 
\end{remark}

\section{A better base ring}

In many cases we can't use Theorem \ref{tammy2} and in many other situations we have the problem that vector spaces of high dimensions with an inner
product are often isotropic. In practice the modules will have a large length as an $R$-module and this is caused by the fact that  $A^{\dagger}/A$ is
an $A$-module, not only an $R$-module. 

In this section let $(R,\pa)$ be a complete discrete valuation ring and let $(A,\m)$ be a local order over $R$. We will find a nice ring between $R$
and $A$ that can be used instead of $R$.

\begin{lemma} \label{nsw}
There is a unique $R$-subalgebra of $A$, say $T$, such that the map $\varphi: T  \to A/\m$ has kernel $\pa T$ and image $(A/\m)_s$, the separable
closure of $R/\pa$ inside $A/\m$. This ring $T$ has the following additional properties:
\begin{enumerate}
\item
$T$ is free over $R$ of rank equal to $[A/\m:R/\pa]_s$, the separability degree of $A/\m$ over $R/\pa$; 
\item
$T$ is a complete discrete valuation ring with maximal ideal $\pa T$; 
\item
$Q(T)$ is finite \'etale over $Q(R)$;
\item
$A$ is an order over $T$.
\end{enumerate}
\end{lemma}
\begin{proof}

We will first construct $T$. It follows from Theorem \ref{ks} that $A$ is complete with respect to $\m$. 
 
Now pick $\alpha \in A$ such that $(A/\m)_s=(R/\pa)[\overline{\alpha}]$. Let $f \in
R[x]$ be monic of degree $[A/\m:R/\pa]_s$ with $\overline{f}(\overline{\alpha})=0$. By construction we have $f(\alpha) \in \m$
and $f'(\alpha) \in A^*$ (by the separability). As $A$ is complete with respect to $\m$, we can apply Hensel's Lemma (\cite{EI}, Theorem 7.3) to find
a unique $\beta \in \alpha+\m$ such that $f(\beta)=0$. Now let $T=R[\beta]$. By construction the image under $\varphi$ is $(A/\m)_s$. As $\m
\cap R=\pa$ (\cite{AT}, Corollary 5.8) it follows that $\pa T \subset \mathrm{ker}(\varphi)$. Now consider the surjective map $\overline{\varphi}:
T/\pa T \to (A/\m)_s$. As the dimensions satisfy $\mathrm{dim}_{R/\pa}(T/\pa T) \leq [A/\m:R/\pa]_s$, the map is injective as well.

Now we will prove the four properties for any $T'$ satisfying the definition in the lemma. 

i. Notice that $R$ is a principal ideal domain and as $A$ is torsion free over $R$, $A$ is free over $R$. By assumption
$T'/\pa T'\cong (A/\m)_s$, and hence $T'$ is free over $R$ of rank $[A/\m:R/\pa]_s$. 

ii. As $T'/\pa T'\cong (A/\m)_s$, a field, it follows that $\pa T$ is a maximal and principal ideal. Theorem 7.2 from \cite{EI} says that $T'$ is
complete with respect to $\pa T'$ and hence local. As $T$ is a regular local ring, it is a domain by Corollary 10.14 from \cite{EI}. As the maximal
ideal is principal, it is a complete discrete valuation ring by \cite{AT},
Proposition 9.2. 

iii. and iv. We know that $Q(A)$ is finite \'etale over $Q(R)$. It follows that $Q(A)$ is a finite product of finite separable field extensions over $Q(R)$. By exactness of localization and Theorem \ref{quotient} we have the inclusions 
\begin{eqnarray*}
Q(R)=R \otimes_R Q(R) \subseteq Q(T')=T' \otimes_R Q(R) \subseteq Q(A)=A \otimes_R Q(R).
\end{eqnarray*}
We see that $Q(T')$ is a separable field extension of $Q(R)$.
This shows that $Q(T')$ is finite \'etale over $Q(R)$ and that $A$ is an order over $T'$. 

We will now prove that $T$ is unique. Suppose we have another $T'$ which satisfies the defining properties. Consider the map $\varphi': T' \to (A/\m)_s$ which has kernel $\pa T'$. By completeness of $T'$ at $\pa T'$ it follows that there is a unique $\alpha' \in \varphi'^{-1}(\overline{\alpha}) \subset \alpha+\m$ satisfying $f(\alpha')=0$. By uniqueness of $\alpha$ it follows that $\alpha=\alpha' \in T'$. Hence $T' \subseteq T$. Now apply Lemma 7.4 from \cite{HAG} Chapter II to see that $T'=T$. 

\end{proof}

Let $T$ be as in the above lemma. We will now prove some more properties. We let $A_R^{\dagger}$ respectively $A_T^{\dagger}$ be the trace duals with respect to $R$ respectively $T$. Similarly, $T_R^{\dagger}$ is the trace dual of $T$ with respect to $R$. 

\begin{lemma} \label{os}
The following properties hold.
\begin{enumerate}
\item
We have $\Delta(T/R) \in R^*$;
\item
$T_R^{\dagger}=T$;
\item
$A_T^{\dagger}=A_R^{\dagger}$.
\end{enumerate}
\end{lemma}
\begin{proof}
i. We have the following commutative diagram: 
\[
\xymatrix{ 
T \ar[d]_{\mathrm{Tr}_{T/R}} \ar[r] & T/\pa T \ar[d]^{\mathrm{Tr}_{T/\pa T / R/\pa}} \\
R \ar[r] & R/\pa.
}
\] 
By definition of $T$ the extension $T/\pa T \supset R/\pa$ is separable, hence has nonzero discriminant. Since the discriminant behaves well with respect to tensoring, this shows that $\Delta(T/R) \in R^*$. 

ii. If $e_1,\ldots, e_n$ is a basis of $T$ over $R$, then it follows that $(\mathrm{Tr}_{T/R}(e_i e_j)_{ij})$ is invertible over $R$ and it follows directly that $T_R^{\dagger}=T$. 

iii. We find using ii for $x \in Q(A)$:
\begin{eqnarray*}
x \in A_R^{\dagger} &\iff& \mathrm{Tr}_{Q(A)/Q(R)}(xA) \subseteq R \\
&\iff& \mathrm{Tr}_{Q(T)/Q(R)} \left( \mathrm{Tr}_{Q(A)/Q(T)}(xA) \right) \subseteq R \\
&\iff& \mathrm{Tr}_{Q(T)/Q(R)} \left(T \cdot \mathrm{Tr}_{Q(A)/Q(T)}(xA) \right) \subseteq R \\
&\iff& \mathrm{Tr}_{Q(A)/Q(T)}(xA) \subseteq T_R^{\dagger}=T \\
&\iff& x \in A_T^{\dagger}. 
\end{eqnarray*}
\end{proof}

We directly see that for a $T$-module $M$ of finite length we have 
\begin{eqnarray*}
\mathrm{lr}(M) &=& \sum_{i=0}^{\infty} \left( \pa^r M \cap M[\pa^r] \right) \\
&=& \sum_{i=0}^{\infty} \left( \pa^r T M \cap M[\pa^r T] \right).
\end{eqnarray*}
This means that the lower root with respect to $R$ is the same as with respect to $T$ for such a module.

\begin{lemma} \label{cs}
Let $T_1 \subseteq T_2$ be a local rings with maximal ideal $\m_1$ respectively $\m_2$ such that $\m_2 \cap T_1=\m_1$. Suppose
that $[T_2/\m_2:T_1/\m_1]<\infty$. Let $M$ be an $T_2$-module of finite length. Then $M$ has finite length over $T_1$ and we have
\begin{eqnarray*}
\mathrm{length}_{T_1}(M)=[T_2/\m_2:T_1/\m_1] \cdot \mathrm{length}_{T_2}(M).
\end{eqnarray*}
\end{lemma}
\begin{proof}
This follows from the fact that $\mathrm{length}_{T_1}(T_2/\m_2)=[T_2/\m_2:T_1/\m_1]$.  
\end{proof}

For a $T$-module $M$ of finite length we find $\mathrm{length}_R(M)= \mathrm{length}_T(M) \cdot [A/\m:R/\pa]_s$.

Then we have the following commutative diagram, where $^-$ stands for the reduction module $A$, $T$ or $R$:
\[
\xymatrix{ 
A^{\dagger}/A \times A^{\dagger}/A \ar[rrrr]^{(\overline{x},\overline{y}) \mapsto \overline{\mathrm{Tr}_{Q(A)/Q(T)}(xy)}}
\ar[drrrr]_{\varphi} &&&& Q(T)/T \ar[d]^{\overline{t}
\mapsto \overline{\mathrm{Tr}_{Q(T)/Q(R)}(t)}}\\
                           &&&& Q(R)/R
}
\]
where $\varphi(\overline{x},\overline{y})=\overline{\mathrm{Tr}_{Q(A)/Q(R)}(xy)}$.

\begin{example}
 Let $f(x)=x^4+25x^3+92x^2+89x+34 \in \Z_3[x]$. Let $A=\Z_3[\alpha]$ where $\alpha$ is a zero of $f$. Then $A$ is an order as $f$ is irreducible
over $\Z$ and $A \cong \Z[x]/(f(x)) \otimes_{\Z} \Z_3$. The ring $A$ has just one prime ideal above $3$, namely $(3,\alpha^2+2\alpha+2)$, with
residue field $\F_9$. As there is a unique unramified extension of $\Z_3$ of degree $2$, we know that $\Z_3[i] \subset A$ and this is a better ring
to work over. We have over $\Z_3$ that $A^{\dagger}/A \cong \left(\Z_3/3^2\Z_3\right)^2$ (actually, the form is anisotropic, but one needs a
calculation to see this). Over $\Z_3[i]$ we find $A^{\dagger}/A \cong \Z_3[i]/3^2 \Z_3[i]$ and by Lemma \ref{or} and Theorem
\ref{main5} we know that the integral closure is given by the lower root. 
\end{example}

\section{Using the better base ring} \label{beba}

Let $A$ be an order over a discrete valuation ring $(R,\pa)$. Let $C$ be any $A$-module that has finite length as an $R$-module. Then it has finite
length as $A$-module (\cite{EI}, Theorem 2.13). By Theorem 2.13 from \cite{EI} we have $C \cong_A \bigoplus_{\m \in \mathrm{MaxSpec}(A)} C_{\m}$. We
will now focus on such a factor $C_{\m}$ as $R$-module, which still has finite length over
$R$ and $A_{\m}$. We have
\begin{eqnarray*}
 C_{\m} \cong_R \bigoplus_{i \geq 1} (R/\pa^i)^{n(i,\m)}.
\end{eqnarray*}
We claim that $[A/\m:R/\pa]$ divides  $n(i,\m)$. To see this notice that \[\left( \left( \pa^{i-1}C_{\m} \right)[\pa] +\pa^i C_{\m} \right)/\pa^i
C_{\m} \cong_R
(R/\pa)^{n(i,\m)}.\] But the left hand side is an $A$-module, so by Lemma \ref{cs} we know that $[A/\m:R/\pa]$ divides $n(i,\m)$. 

We can apply the above to $A^{\dagger}/A$ and $\overline{A}/A$.
We can finally prove a local version of Theorem \ref{3} from the introduction. 

\begin{theorem} \label{fin}
 Let $(R,\pa)$ be a discrete valuation ring and let $A$ be an order over $R$. Let $\m$ be a maximal ideal of $A$ and let $B=\left(A^{\dagger}/A
\right)_{\m} \cong_R \bigoplus_{i \geq 1} (R/\pa^i)^{n(i,\m)}$. Suppose that the following conditions are satisfied: 
\begin{enumerate}
\item
$\mathrm{char}(A/\m)=0$ or $\mathrm{char}(A/\m)> \mathrm{dim}_{R/\pa}(B/\pa B)=\sum_{i \geq 1} n(i,\m)$;
\item
There exist $i_1,i_2 \in \Z_{\geq 1}$ such that 
\begin{itemize}
\item
$i_1 \not \equiv i_2 \bmod{2}$;
\item
$n(i,\m)=0$ for all $i \not \in \{1,i_1,i_2\}$;
\item
$n(i,\m) \in \{0,[A/\m:R/\pa]\}$ for $i \in \{i_1,i_2\}$.
\end{itemize}
\end{enumerate}
Then we have $\left(\overline{A}/A \right)_{\m}=\mathrm{lr}\left(
(A^{\dagger}/A)_{\m} \right)$
\end{theorem}
\begin{proof}
 The idea of the proof is to complete and then work over the better base ring from the previous section. 
Lemma \ref{cheata} and Lemma \ref{sor} allow us to assume that $A$ is a local order with maximal ideal $\m$ over a complete discrete valuation ring.
In this case we can use our
better base ring $T$
as in Lemma \ref{nsw}.
If all $n(i,\m)$ are zero, then the conclusion of the theorem is correct. Otherwise it follows from assumption i that $\mathrm{char}(A/\m)=0$ or
$\mathrm{char}(A/\m)>[A/\m:R/\pa]=d$. Both of these assumptions imply that $A/\m \supseteq R/\pa$ is separable. Now let $A'$ be any $T$-order between
$A$ and $\overline{A}$. Notice that $Q(T)$ is a finite \'etale $Q(R)$-algebra (Lemma \ref{nsw}), and hence it directly follows that $A'$ is an
order over $R$. Theorem \ref{tammy2} and condition i imply that $A'$ is tame at $\pa$. As $R/\pa \subseteq T/\pa T \subseteq A'/\pa A'=A'/\pa T A'$,
it follows by 
Lemma \ref{hu} that $A'$ is tame at $\pa T$. 

The result now follows directly from Theorem \ref{nano} if we can show that the induced form is quasi-anisotropic. We know that
$\left( \left( \pa^{i-1}TB \right)[\pa T] +\pa^i T B \right)/\pa^i TB \cong_R
(R/\pa)^{n(i,\m)}$. Using Lemma \ref{cs} and the separability of $A/\m \supseteq R/\pa$ it follows that $B \cong_T
(T/\pa T)^{n(1,\m)/d} \oplus (T/\pa^{i_1}T)^{n(i_1,\m)/d} \oplus (T/\pa^{i_2}T)^{n(i_2,\m)/d}$. Now use condition ii, Lemma \ref{bra}, Lemma
\ref{or}
to see that our form is indeed quasi-anisotropic.
\end{proof}

We can now prove the last theorem of the introduction. 

\begin{proof}[Proof of Theorem \ref{3}.]
Reduce to the local case by Lemma \ref{floop} and use Theorem \ref{fin}.
\end{proof}

\section{Acknowledgements} 

I would like to thank Professor Hendrik Lenstra for essentially coming up with all the new theory and for helping me to write this article. Without
his help this article wouldn't be possible.

\end{document}